\documentclass[11pt]{article}
\usepackage[utf8]{inputenc}
\usepackage{hyperref}
\usepackage{amsmath,amssymb,amsthm}
\usepackage{amsfonts,tikz}
\usepackage[margin=1in]{geometry}
\usepackage{mathtools}
\usepackage{mathrsfs}
\usepackage{enumitem}
\usepackage[backend=biber]{biblatex}
\usetikzlibrary{positioning,automata}

\newcommand{\lm}{\lambda}

\newcommand{\tm}{\widetilde{m}}
\newcommand{\m}[1]{\ (\text{mod }#1)}
\newcommand{\G}{\mathcal{G}}

\newtheorem{lemma}{Lemma}[section]
\newtheorem{cor}{Corollary}[section]
\newtheorem{theorem}{Theorem}[section]

\newtheorem{proposition}{Proposition}[section]
\theoremstyle{definition}
\newtheorem{definition}{Definition}[section]

\title{A Graph Polynomial from Chromatic Symmetric Functions}
\author{William Chan, Logan Crew\footnote{Department of Combinatorics \& Optimization, University of Waterloo, Waterloo, ON, N2L 3E9.\newline  Emails:  w47chan@uwaterloo.ca, lcrew@uwaterloo.ca
}
}
\date{\today}

\begin{document}

\maketitle

\begin{abstract}

This paper describes how many known graph polynomials arise from the coefficients of chromatic symmetric function expansions in different bases, and studies a new polynomial arising by expanding over a basis given by chromatic symmetric functions of trees. 
    
\end{abstract}

\newline

\begin{section}{Introduction}
The chromatic symmetric function $X_G$ of a graph $G$ was introduced by Stanley in the 1990s \cite{STANLEY1995166} as a generalization of the chromatic polynomial, defined as
\[X_G := \sum_{\substack{\kappa : V(G) \to\mathbb{N}\setminus\{0\}\\\kappa\text{ proper colouring}}}x_{\kappa(v_1)}\dots x_{\kappa(v_{|V(G)|})}.\]

In particular, while $[x^k]\chi_G$ counts the number of proper $k$-colourings of $G$ for each $k$, if $\lm = (\lm_1, \dots, \lm_k)$ with $\lm_1 \geq \dots \geq \lm_k > 0$ is such that $\lm_1 + \dots + \lm_k = |V(G)|$ is an integer partition, then when expanding in the monomial symmetric function basis the coefficient $[m_{\lm}]X_G$ gives (up to a constant factor) the number of proper $k$-colourings of $G$ such that colour $i$ is used $\lm_i$ times. The study of $X_G$ is currently an active area of research due to its connections to other areas of mathematics, including algebraic geometry \cite{brosnan, sazdanovic} and knot theory \cite{chmutov, noble}.

If $l(\lm)$ is the number of distinct integers in the partition $\lm$, Stanley showed that when the chromatic symmetric function of a graph $G$ is expanded over the elementary symmetric function basis $\{e_{\lm}\}$, mapping each $e_{\lm} \rightarrow x^{l(\lm)}$ yields a polynomial $p(x)$ such that $[x^k]p(x)$ is the number of acyclic orientations of $G$ with $k$ sinks (\cite{STANLEY1995166} Theorem 3.3); this polynomial and its generalizations have recently been further studied by Hwang et. al. \cite{acyclic}. Similarly, expanding $X_G$ into the power-sum symmetric function basis $\{p_{\lm}\}$ and setting each $p_{\lm} \rightarrow x^{l(\lm)}$ yields $\chi_G(x)$, and doing the same with the augmented monomial symmetric function basis $\tm_{\lm}$ yields a polynomial $q(x)$ such that $[x^k]q(x) = \chi_G(k)$.

In recent work \cite{crew2020deletioncontraction}, Spirkl and the second author extended $X_G$ to graphs with vertices weighted by positive integers to give the function a deletion-contraction relation, and Aliniaeifard, Wang, and van Willigenburg \cite{aliniaeifard2021extended} showed that if $\{G_i\}_{i=1}^{\infty}$ is a family of connected, vertex-weighted graphs such that $G_i$ has weight $i$, then their chromatic symmetric functions are algebraically independent, and so may be used to form a basis $\{X_{G_{\lm}}\}$ by letting $G_{\lm}$ be the disjoint union of $G_{\lm_1}, \dots, G_{\lm_k}$; for instance, the aforementioned $p$ and $e$ bases arise in this way where $G_i$ is a single vertex of weight $i$ and an unweighted clique of size $i$ respectively.

In this work, we show that for every (possibly weighted) graph $G$, if $\{T_i\}_{i=1}^{\infty}$ is a family of unweighted trees such that $T_i$ has $i$ vertices, the polynomial obtained by mapping $X_{T_{\lm}} \rightarrow x^{l(\lm)}$ in $X_G$ is independent of the choice of the $T_i$. We call this polynomial the \emph{tree polynomial of $G$}, and we study its properties, noting especially a kind of duality with the chromatic polynomial for vertex-weighted graphs.

In Section 2 we provide the necessary background in symmetric functions and colouring. In Section 3 we introduce the tree polynomials described above, and give a combinatorial interpretation for their coefficients. In Section 4 we find further relationships between the chromatic polynomial and tree polynomial, as well as an interpretation for the evaluation of the tree polynomial on positive integers. Finally, in Section 5 we investigate further directions for research.
\end{section}
%Add more sections here
\begin{section}{Background}
We take the natural numbers to contain $0$. We define the interval $[a, b] = \{a, a+1, \dots, b\}$ where $a, b \in \mathbb{N}$. 

A \textit{partition} $\lambda$ of $n \in \mathbb{N}\setminus\{0\}$ is a weakly decreasing tuple $(\lambda_1, \dots, \lambda_k)$ of positive integers such that $\sum_{i=1}^{k}\lambda_i = n$. We write $\ell(\lambda) = k$ to be the \emph{length} of $\lm$ and $|\lambda| = n$ to be the \emph{size} of $\lambda$. We write $\lambda \vdash n$ to mean $\lambda$ is a partition of $n$. When we say a partition (without specifying its size) we mean a partition of some $n\in\mathbb{N}\setminus\{0\}$. 

We take all graphs $G$ to be simple. In particular, we take the edge-contraction graph $G/e$ with $e \in E(G)$ to not have any loops or multi-edges. A \emph{weighted graph} (as defined in \cite{crew2020deletioncontraction}) will be denoted as $(G, \omega)$ where $\omega : V(G) \to \mathbb{N}\setminus\{0\}$ is a weight function. We define the contraction of a weighted graph $(G, \omega)$ by an edge $e=v_1v_2 \in E(G)$ to be the graph $(G/e, \omega/e)$, where $(\omega/e)(v) = \omega(v_1) + \omega(v_2)$ if $v$ is the vertex obtained from contracting $v_1v_2$ and $(\omega/e)(v) = \omega(v)$ otherwise. We say the \emph{total weight} of $(G,\omega)$ is $\sum_{v\in V(G)}\omega(v)$, and the \emph{excess weight} of $(G,\omega)$ is $(\sum_{v \in V(G)}\omega(v)) - |V(G)|$, that is, the total weight minus the number of vertices. 

We will always assume there is an implicit total ordering on the edges of $G$ (implicit because the results that follow do not depend on the choice of ordering). Where $e$ and $f$ are edges of $G$, we say $e < f$ if $e$ is less than $f$ in this ordering. We let $C(G)= \{G_1, \dots, G_k\}$ denote the set of connected components of $G$, where each component $G_i$ is a subgraph of $G$. We write $H \le G$ where $H$ and $G$ are graphs to mean $H$ is a subgraph of $G$. A spanning subgraph $H \le G$ is called \textit{maximally connected} if $C(H) = \{H_1, \dots, H_k\}$ and $C(G) = \{G_1, \dots, G_k\}$ with $\{V(H_1), \dots, V(H_k)\} = \{V(G_1), \dots, V(G_k)\}$. Note that this is equivalent to $E(G)\setminus E(H)$ containing no bridges for $H$. 

If $F$ is a spanning forest of $G$ we say an edge $e \in E(G)\setminus E(F)$ is \emph{externally active} if $F \cup e$ is not a forest and $e$ is the minimal edge in the unique cycle in $F \cup e$. We say $F$ is \emph{internal} if no edges $e \in E(G)\setminus E(F)$ are externally active. For $S \subseteq E(G)$, we let $\G(S) = (V(G), S)$. That is, $\G(S)$ is the spanning subgraph induced by $S$.
\begin{definition}
Let $G$ be a graph. An edge set $S \subseteq E(G)$ is a \emph{broken circuit} if it contains all the edges of some cycle of $G$ other than its smallest one.
\end{definition}
\begin{definition}
The \emph{broken circuit complex} $B_G$ is the set of all $S \subseteq E(G)$ which do not contain (all the edges of) a broken circuit.
\end{definition}
We note that the elements of $B_G$ are subsets of $E(G)$ and thus $B_G$ has an induced partial order arising from the partial order on $E(G)$ with respect to inclusion. From the definition of $B_G$, we note that for each element $F \in B_G$ that $\G(F)$ is acyclic and hence a forest, as the existence of a cycle means the existence of the broken circuit contained in that cycle. Moreover, we note that for each $S \in B_G$, it follows that $\G(S)$ is internal, as an externally active edge would imply the existence of a broken circuit. Conversely, it is clear that the edge set of every internal forest is an element of $B_G$, so $B_G$ consists precisely of the edge sets of internal forests of $G$.

For $S \in B_G$ and $C(\G(S)) = \{H_1, \dots, H_k\}$, we write $\lambda(S)$ to be the integer partition with parts $|V(H_1)|, \dots, |V(H_k)|$. 
\begin{lemma}
$S \in B_G$ is (inclusion-wise) maximal if and only if $\G(S)$ is maximally connected.
\end{lemma}
\begin{proof}
If $\G(S)$ is maximally connected then $E(G)\setminus S$ contains no bridges for $\G(S)$. So if $E(G) \supseteq \pi \supset S$ then $\G(\pi)$ contains a cycle. But for $\sigma \in B_G$, we must have $\G(\sigma)$ a forest. So $\pi \not\in B_G$ which shows $S \in B_G$ is inclusion wise maximal. 

Conversely, suppose $\G(S)$ is not maximally connected with $S \in B_G$. Then there exists a set of bridges $\{f_1, \dots, f_k\} \in E(G)\setminus S$ for $\G(S)$ with each $f_1 < \dots < f_k$ in the implicit total ordering on $E(G)$. Consider $S' := S \cup \{f_1\}$. Suppose $S'$ contains broken circuit $W$, which is the cycle $C$ with the smallest edge removed. Note that because $S \in B_G$, $W$ must contain $f_1$. Suppose the smallest edge in $C$ is $\tilde{f}$. As $\tilde{f} \not\in W$ and $f_1 \in W$, it follows that $\tilde{f} \ne f_1$. We note that $S \cup \{\tilde{f}\} = (S' \cup \{\tilde{f}\}) \setminus \{f_1\}$ contains no cycle as $S' \cup \{\tilde{f}\}$ contains exactly one cycle and $f_1$ is in that cycle. So $\tilde{f}$ is a bridge for $\G(S)$ which contradicts the minimality of $f_1$. 
\end{proof}
With this lemma as motivation, we say $F$ is an \emph{internal preserving forest} if $E(F)$ is in $B_G$ and is maximal in $B_G$.
\begin{definition}
Let $G = (V, E)$ be a graph. We say $S \subseteq E$ is a \emph{$k$-cutset} if $|C((V, E\setminus S))| = k$, that is, if removing $S$ yields $k$ connected components.
\end{definition}
Let $W$ be a vector space and let $A \subseteq W$ be linearly independent. For $f \in \text{span}(A)$ and $a \in S$, we write $[a]f$ to mean the coefficient attached to $a$ when uniquely writing $f$ as a linear combination of elements in $S$. 

For $u, v \in V(G)$ we write $u \sim v$ to mean $u$ is adjacent to $v$, i.e. $uv \in E(G)$. A \emph{proper $k-$colouring} of $G$ is a function $f : V(G) \to \{1, \dots, k\}$ such that $u \sim v \implies f(u) \ne f(v)$. A \emph{proper colouring} is a function $g : V(G) \to \mathbb{N}\setminus\{0\}$ is a proper $k-$colouring for some $k\in\mathbb{N}\setminus\{0\}$.
Let $\chi_G(x)$ be the \emph{chromatic polynomial} of $G$, defined such that for each $k \in \mathbb{N} \setminus \{0\}$, $\chi_G(k)$ is the number of proper $k-$colourings of $G$.
\begin{proposition}
$\sum_{i}[x^i]\chi_G = 0$ if $E(G) \ne \emptyset$ and $1$ if $E(G) = \emptyset$
\end{proposition}
\begin{proof}
If $E(G) = \emptyset$ then $\chi_G = x^{|V(G)|}$ and we're done. Suppose $|E(G)| \geq 1$, and pick $e \in E(G)$. By the well-known deletion-contraction relation of the chromatic polynomial,
\[\chi_G = \chi_{G\setminus e} - \chi_{G/e}.\]
Note that $E(G\setminus e) = \emptyset$ if and only if $E(G) = \{e\}$ if and only if $E(G/e)=\emptyset$. So $\sum_{i}[\chi_{G\setminus e}]x^i = \sum_{i}[\chi_{G/e}]x^i$ and the result follows. 
\end{proof}

Let $\Lambda$ be the algebra of symmetric functions over $\mathbb{Q}$ in variables $x_1, x_2, \dots$. Let $\{b_\lambda : \lambda\vdash n\ge 0\}$ be a basis for $\Lambda$. We say $\{b_\lambda : \lambda\vdash n\ge 0\}$ is a \emph{multiplicative basis} if each $b_\lambda = b_{\lambda_1}\dotsb_{\lambda_{\ell(\lambda)}}$. We define the \textit{power-sum} basis $\{p_\lambda : \lambda\vdash n\ge0\}$ to be the multiplicative basis with each $p_n = \sum_{i=1}^{\infty}{x_i}^n$. For more information regarding symmetric functions, we defer to \cite{ecv2}. 

For a family of weighted graphs $\{(G_n,\omega_n)\}_{n=1}^{\infty}$ and a partition $\lambda$, we define the weighted graph $(G_\lambda, \omega_\lambda)$ to be the disjoint union of weighted graphs
\[(G_\lambda,\omega_\lambda) := (G_{\lambda_1}, \omega_{\lambda_1}) \cup \dots \cup (G_{\lambda_{\ell(\lambda)}}, \omega_{\ell(\lambda)}).\]
Similarly, if $\{G_n\}_{n=1}^{\infty}$ is a family of graphs, we define $G_\lambda$ to be the disjoint union of graphs
\[G_\lambda= G_{\lambda_1} \cup \dots \cup G_{\lambda_\ell(\lambda)}.\]

 In \cite{STANLEY1995166}, the chromatic symmetric function of $G$ with $V(G) = \{v_1, \dots, v_{|V(G)|}\}$ is defined as
\[X_G := \sum_{\substack{\kappa : V(G) \to\mathbb{N}\setminus\{0\}\\\kappa\text{ proper colouring}}}x_{\kappa(v_1)}\dots x_{\kappa(v_{|V(G)|})}\]
as a symmetric function analogue to the chromatic polynomial. With this definition, $X_G$ is homogeneous of degree $|V(G)|$ so the function does not satisfy a simple deletion-contraction relation, as contraction reduces the number of vertices. Spirkl and the second author thus generalized the chromatic symmetric function to the weighted chromatic symmetric function defined by \cite{crew2020deletioncontraction}
\[X_{(G, \omega)} \sum_{\substack{\kappa : V(G) \to\mathbb{N}\setminus\{0\}\\\kappa\text{ proper colouring}}}x_{\kappa(v_1)}^{\omega(v_1)}\dots x_{\kappa(v_{|V(G)|})}^{\omega(v_{|V(G)|})}\]
which does satisfy the deletion contraction relation (\cite{crew2020deletioncontraction} Lemma 2) for $e \in E(G)$,
\[X_{(G,\omega)} = X_{(G\setminus e, \omega)} - X_{(G/e, \omega/e)}.\]
Aliniaeifard, Wang, and van Willigenburg showed that one can form bases of $\Lambda$ which come from weighted chromatic symmetric functions, generalizing the work of \cite{cho} for unweighted graphs:
\begin{theorem}[\cite{aliniaeifard2021extended} Theorem 5.4]
Let $\{(G_n, \omega_n)\}_{n=1}^{\infty}$ be a family of connected weighted graphs such that the total weight of each $G_n$ is $n$. Then $\{X_{(G_\lambda, \omega_\lambda)} : \lambda\vdash n\ge0\}$ is a multiplicative basis for $\Lambda$.
\end{theorem}
Such bases for $\Lambda$ are called \emph{chromatic bases}. For instance, among the c five classical symmetric function bases \cite{ecv2}, the $m-$, $e-$, and $p-$ basis are chromatic, while the $h-$ and $s-$ basis are not \cite{chobasis, crew2020deletioncontraction}.
\begin{cor}
Let $\{T_n\}_{n=1}^{\infty}$ be a family of trees with each $T_n$ having $n$ vertices. Then 
\[\{X_{T_\lambda} : \lambda\vdash n\ge0\}\] 
is a multiplicative basis for $\Lambda$.
\end{cor}
We will call such bases \emph{tree bases}. In particular, $\{X_{P_\lambda} : \lambda\vdash n\ge0\}$, $\{X_{S_\lambda} : \lambda\vdash n\ge0\}$, and $\{X_{C_\lambda} : \lambda\vdash n\ge0\}$ where $P_n, S_n, C_n$ are the paths, stars, and cycles on $n$ vertices respectively are multiplicative bases for $\Lambda$. In the same paper, it was found that there is a peculiar relation between certain chromatic bases and the power-sum basis.
\begin{theorem}\label{thm:chromrec}[Chromatic Reciprocity, \cite{aliniaeifard2021extended} Theorem 6.3]
The maps $p_\lambda\to X_{P_\lambda}$ and $p_\lambda \to X_{S_\lambda}$ are automorphisms of $\Lambda$ that are also involutions.
\end{theorem}
\end{section}
\begin{section}{Tree Coefficients}
We begin with some definitions. By Corollary 2.1 we can give the following definition,
\begin{definition}
For a symmetric function $f = \sum_{\lambda}a_\lambda X_{P_\lambda}$, define $\tau_f(x)$ to be the polynomial with
\[[x^k]\tau_f = \sum_{\substack{\lambda\\\ell(\lambda)=k}}a_\lambda.\]
\end{definition}
\begin{definition}
For a symmetric function $f = \sum_{\lambda}a_\lambda p_\lambda$, define $\chi_f(x)$ to be the polynomial with
\[[x^k]\chi_f = \sum_{\substack{\lambda\\\ell(\lambda)=k}}a_\lambda.\]
\end{definition}
Note that $\tau : \Lambda \to \mathbb{Q}[x]$ and $\chi : \Lambda\to\mathbb{Q}[x]$ are the linear transformations which extend $X_{P_\lambda} \to t^{\ell(\lambda)}$ and $p_\lambda \to t^{\ell(\lambda)}$. In general,
\begin{definition}
Let $B = \{b_\lambda : \lambda\}$ be a basis for $\Lambda$. The \emph{$B-$polynomial} of $f = \sum_{\lambda}a_\lambda b_\lambda \in \Lambda$ is the polynomial $p(x)$ with
\[[x^k]p(x) = \sum_{\substack{\lambda\\\ell(\lambda)=k}}a_\lambda\]
\end{definition}
\begin{lemma}
Let $f_1, f_2$ be symmetric functions. Let $B$ be a multiplicative basis for $\Lambda$ and let $\theta_{f_i}$ denote the $B-$polynomial of $f_i$. Then $\theta_{f_1f_2} = \theta_{f_1}\theta_{f_2}$. That is $\theta : \Lambda \to \mathbb{Q}[t]$ is also an algebra homomorphism.
\end{lemma}
\begin{proof}
For $k \in \mathbb{N}$,
\begin{align*}
[x^k]\theta(f_1f_2) &= \sum_{\substack{\lambda\\\ell(\lambda)=k}}[b_\lambda]f_1f_2\\
&= \sum_{t=1}^{k}\sum_{\substack{\alpha\\\ell(\alpha)=t}}\sum_{\substack{\beta\\\ell(\beta)=k-t}}[b_\alpha]f_1[b_\beta]f_2\\
&= \sum_{t=1}^{k}\left(\sum_{\substack{\alpha\\\ell(\alpha)=t}}[b_\alpha]f_1\right)\left(\sum_{\substack{\beta\\\ell(\beta)=k-t}}[b_\beta]f_2\right)\\
&= \sum_{t=1}^{k}[x^t]\theta_{f_1}[x^{k-t}]\theta_{f_2}\\
&= [x^k]\theta_{f_1}\theta_{f_2}
\end{align*}
The result follows
\end{proof}
\begin{lemma}
Let $n \in \mathbb{N}$. Then the set of polynomials $x^k(x-1)^{n-k}$ for $k \in [0,n]$ are linearly independent in $\mathbb{Q}[x]$.
\end{lemma}
\begin{proof}
Note that each $x^k(x-1)^{n-k} = \sum_{m=0}^{n-k}(-1)^{n-k-m}\binom{n-k}{m}x^{m+k}$. Suppose $\sum_{k=0}^{n}r_kx^k(x-1)^{n-k} = 0$. Note that $r_0 = 0$ because it is the only $x^k(x-1)^{n-k}$ which has a constant term. Suppose that $r_0 = \dots = r_t = 0$ for some $t \in [0, n-1]$. Note then that $r_{t+1} = 0$ as it is the only remaining $x^k(x-1)^{n-k}$ with a term of $x^{t+1}$. It follows that $r_k = 0$ for each $k \in [0,n]$.
\end{proof}
This lemma will allows us to prove results (Proposition 4.2, Theorem 4.1) using the chromatic equivalence of trees.
\begin{proposition}
Let $G$ be a graph. Then $\chi_G = \chi_{X_G}$.
\end{proposition}
\begin{proof}
Follows from the formulae (\cite{STANLEY1995166}, Theorems 2.9 and 1.3 respectively)%Stanley thms 2.9, 1.3 respectively
\[X_G = \sum_{S \in B_G}(-1)^{|S|}p_{\lambda(S)}\]
and
\[\chi_G = \sum_{S\in B_G}(-1)^{|S|}x^{|V(G)|-|S|}\]
Indeed, as each $(V(G), S)$ is a forest, the number of components $k$ of $(V(G), S)$ satisfies $|S| = |V(G)| - k$ or $k = |V(G)| - |S|$ and hence $\ell(\lambda(S)) = |V(G)| - |S|$.
\end{proof}
We let $\tau_G := \tau_{X_G}$ and $\tau_{(G,\omega)} := \tau_{X_{(G,\omega)}}$. As a result of the above proposition, we more generally adopt the notation that if $\{b_\lambda : \lambda\vdash n\ge0\}$ is a multiplicative basis and $\theta : \Lambda \to \mathbb{Q}[t]$ is the map $b_\lambda\mapsto t^{\ell(\lambda)}$, we say $\theta_G := \theta_{X_G}$ and $\theta_{(G,\omega)} = \theta_{X_{(G,\omega)}}$. Indeed, $\chi_G$ the chromatic polynomial and $\chi_G$ the power-sum basis polynomial are the same by the above proposition. 
\begin{proposition}
Let $F$ be a forest of order $n$ with $m$ parts. Then $\tau_F = t^{m}$. 
\end{proposition}
\begin{proof}
Write $X_F = \sum_{\lambda\vdash n}a_\lambda X_{P_\lambda}$. Note that $\chi_{P_\lambda} = x^{\ell(\lambda)}(x-1)^{n-\ell(\lambda)}$. Applying $\chi$ to both sides,
\begin{align*}
x^m(x-1)^{n-m} &= \chi_F= \chi_{X_F}= \sum_{\lambda\vdash n}a_\lambda\chi_{X_{P_\lambda}}=\sum_{\lambda\vdash n}a_\lambda x^{\ell(\lambda)}(x-1)^{n-\ell(\lambda)}
\end{align*}
Because the $x^t(x-1)^{n-t}$ are linearly independent, the result follows by comparing coefficients.
\end{proof}
\begin{cor}
For a tree $T$, $\tau_T = x$.
\end{cor}
\begin{theorem}
Let $\{T_n\}_{n=1}^{\infty}$ be a family of trees with each $T_n$ having $n$ vertices and $\{X_{T_\lambda} : \lambda\vdash n\ge0\}$ its corresponding tree basis. For any weighted graph $(G,\omega)$ and $k\in\mathbb{N}$,
\[\sum_{\substack{\lambda\\\ell(\lambda)=k}}[X_{T_\lambda}]X_{(G,\omega)} = \sum_{\substack{\lambda\\\ell(\lambda)=k}}[X_{P_\lambda}]X_{(G,\omega)}\]
\end{theorem}
\begin{proof}
Write $X_{(G,\omega)} = \sum_{\lambda}a_\lambda X_{P_\lambda} = \sum_{\lambda}b_\lambda X_{T_\lambda}$. Letting $N$ be the total weight of $(G,\omega)$, we obtain
\begin{align*}
\chi_{(G,\omega)} &= \sum_{\lambda}a_\lambda\chi_{X_{P_\lambda}}\\
&= \sum_{\lambda}a_\lambda x^{\ell(\lambda)}(x-1)^{N-\ell(\lambda)}\\
&= \sum_{\lambda}b_\lambda x^{\ell(\lambda)}(x-1)^{N-\ell(\lambda)}\\
&= \sum_{\lambda}b_\lambda \chi_{X_{T_\lambda}}
\end{align*}
The result follows by the linear independence of $x^{\ell(\lambda)}(x-1)^{N-\ell(\lambda)}$ and comparing coefficients.
\end{proof}
With Theorem 3.1 as motivation, we call $\tau_f$ the \emph{tree polynomial} of $f$ since $\tau_f$ is independent of the family of trees chosen. For a graph $G$ we say $\tau_G$ is the tree polynomial of $G$ and $\tau_{(G,\omega)}$ the tree polynomial of $(G,\omega)$. Using the chromatic reciprocity of $p_{\lm}$ and $X_{P_{\lm}}$ (Theorem \ref{thm:chromrec}), we obtain the formula
\begin{theorem}
For a graph $G$ of order $n$,
\[[x^k]\tau_G = (-1)^{n+k}\sum_{m=1}^{k}\binom{n-m}{k-m}[x^m]\chi_G\]
\end{theorem}
\begin{proof}
Writing $X_G = \sum_{\lambda\vdash n}a_\lambda p_\lambda$ we calculate,
\begin{align*}
\sum_{\substack{\lambda\vdash n\\\ell(\lambda)=k}}[X_{P_\lambda}]X_G &= \sum_{\lambda\vdash n}a_\lambda\sum_{\substack{\Delta\vdash n\\\ell(\Delta)=k}}[X_{P_\Delta}]p_\lambda\\
(\text{\cite{aliniaeifard2021extended} Thm 6.3})&=\sum_{\lambda\vdash n}a_\lambda\sum_{\substack{\Delta\vdash n\\\ell(\Delta)=k}}[p_\Delta]X_{P_\lambda}\\
&= \sum_{m=1}^{k}\sum_{\substack{\lambda\vdash n\\\ell(\lambda)=m}}(-1)^{n-k}a_\lambda\binom{n-m}{k-m}\\
&= (-1)^{n-k}\sum_{m=1}^{k}\binom{n-m}{k-m}\sum_{\substack{\lambda\vdash n\\\ell(\lambda)=m}}a_\lambda\\
&= (-1)^{n-k}\sum_{m=1}^{k}\binom{n-m}{k-m}[x^m]\chi_G
\end{align*}
\end{proof}
We may use this formula to interpret the tree polynomial's coefficients. We introduce some further definitions and lemmas.
\begin{definition}
Let $G$ be a graph and $H \le G$. Let $C(H) = \{H_1, \dots, H_m\}$. Define $\nabla_G(H) := \bigcup_{i=1}^{m}\delta_G(H_i)$ where 
\[\delta_G(H_i) = \{uv \in E(G) : u \in V(H_i),\ v \in V(G)\setminus V(H_i)\text{ or }u \in V(G)\setminus V(H_i), v \in V(H_i)\}\]
In other words, $\delta_G(H_i)$ is the cut of $H_i$ calculated in $G$.
\end{definition}
\begin{definition}
Let $G$ be a graph. For a subset of the edges $E$, we let $\G(E)$ denote the spanning subgraph of $G$ induced by $E$. We also define $\ell(E) := |C(\G(E))|$ and $C(G) = C(\G(E))$.
\end{definition}
\begin{definition}
For a graph $G = (V, E)$ and $S \subseteq E$ we write $\nabla_G(S) = \nabla_G(\G(S))$. 
\end{definition}
\begin{definition}
Let $G$ be a graph and $F$ a subforest of $G$. Let $e \in E(G)\setminus E(F)$. If $F \cup e$ produces a cycle, then the cycle is unique and we call it the \emph{fundamental cycle} of $e$.
\end{definition}
\begin{lemma}
Let $G$ be a graph and $\sigma \in B_G$ with $C(\sigma) = \{D_1, \dots, D_k\}$. Let $S = \nabla_G(\G(\sigma))$. If $\hat{e}$ is the minimal edge in $S$ then $\sigma\cup\hat{e} \in B_G$.
\end{lemma}
\begin{proof}
Let $e \in E(G)\setminus \sigma$. If $e$ is externally active in $\sigma\cup\hat{e}$ it must have one end in $D_i$ and one end in $D_j$ with $i\ne j$ (if it didn't, it would have both ends in $D_t$ which would contradict $\G(\sigma)$ having zero external activity). The fundamental cycle of $e$ must then contain $\hat{e}$ which is smaller than $e$. So $e$ cannot be externally active in $\G(\sigma\cup\hat{e})$. The result follows.
\end{proof}
The following lemma is the crux in understanding the tree polynomial's coefficients.
\begin{lemma}
Let $G$ be a graph with $n$ vertices and fix $\pi \in B_G$. Then, $\sum_{\substack{\sigma\supseteq\pi\\\sigma\in B_G}}(-1)^{\ell(\sigma)}$ is $(-1)^{|C(G)|}$ times the number of maximal $T \supseteq \pi$ in $B_G$ such that the minimal edge in $\nabla_G(\pi)$ is not contained in $T$ and for every other edge $e \in \nabla_G(\pi) \cap T$, $e$ is contained in the fundamental cycle of some smaller edge in $\nabla_G(\pi)$.
\end{lemma}
\begin{proof}
Let $m := |\nabla_G(\pi)|$ and $\nabla_G(\pi) = \{e_1, \dots, e_m\}$ where $i < j$ if and only if the edge $e_i$ is less than the edge $e_j$. Construct sets $S_0, S_1, \dots, S_m$ recursively. Let $S_0 := \{\sigma \in B_G : \sigma \supseteq \pi\}$. Say an edge $e \in E(G)\setminus \sigma$ is dead in $\sigma$ if $\sigma \cup \{e\}$ contains a cycle. Let $S$ be a set of edgesets in $B_G$. Say $S$ is $t-$ordered with respect to $\pi$ if for all $\sigma \in S$,
\begin{center}
For $i \in [2,n]$, $e_i$ is either dead or $e_i \in \sigma$ and there exists $j < i$ with $e_j \not\in \sigma$ with $e_i$ in the fundamental cycle of $e_j$. $(*)$
\end{center}
Note that $S_0$ is a $0-$ordered edgeset. It is also an inclusion wise maximal $0-$ordered edgeset in $B_G$ containing $S_0$ because it is $S_0$. Suppose $S_t \in B_G$ is an inclusion-wise maximal $t-$ordered edgeset contained in $S_0$. Construct $S_{t+1}$ by removing $\sigma, \sigma \cup e_{t+1}$ for each $\sigma$ where $e_{t+1}$ is not dead. Note that $\ell(\sigma) \equiv 1 + \ell(\sigma \cup e_{t+1})\m{2}$. Therefore $\sum_{\sigma \in S_{t+1}}(-1)^{\ell(\sigma)} = \sum_{\sigma \in S_t}(-1)^{\ell(\sigma)}$. For the remaining $\sigma \in S_{t+1}$ either $e_{t+1}$ is dead in $\sigma$ or $e_{t+1} \in \sigma$. In the latter case, $\sigma\setminus e_{t+1} \not\in S_t$ as otherwise, $\sigma\setminus e_{t+1}$ would have been an element in $S_t$ with $e_{t+1}$ alive. Hence, $\sigma\setminus e_{t+1}$ violates $(*)$. Let $e_s$ violate $(*)$ in $\sigma\setminus e_{t+1}$. If $e_s$ is alive (that is, $e_s \not\in \sigma\setminus e_{t+1}$ \textit{and} $e_s$ is not dead) then $e_s$ is dead in $\sigma$. If $e_s$ is no longer dead in $\sigma\setminus e_{t+1}$, then it must be the case that $e_{t+1}$ was in the fundamental cycle of $e_s$. Otherwise, $e_s \in \sigma\setminus e_{t+1} \subseteq \sigma$. In particular, $e_s$ satisfies $(*)$ in $\sigma$. So $e_s$ is contained in the fundamental cycle of $e_r$, $r<s$. If $e_s$ is not contained in the fundamental cycle of $e_r$ in $\sigma\setminus e_{t+1}$ then $e_{t+1}$ must have been part of the same fundamental cycle. So $S_{t+1}$ is $t+1$-ordered.
\\\\
To show maximality, consider $\sigma \supseteq \pi$ which satisfy $(*)$ for $t+1$. Then $\sigma$ satisfies $(*)$ for $t$ so that $\sigma \in S_t$. By $(*)$, either $e_{t+1}$ is dead in $\sigma$ or $e_{t+1}$ is in the fundamental cycle of $e_q$, $q < t+1$ in which case $e_q$ is not dead in $\sigma\setminus e_{t+1}$ (as otherwise, $e_q$ would be the missing edge of two cycles in $\sigma$ and thus $\sigma$ would contain a cycle, a contradiction) so $\sigma\setminus e_{t+1} \not\in S_t$. In particular, there exists no $\alpha \in S_t$ with $e_{t+1} \not\in \alpha$ and $e_{t+1}$ alive in $\alpha$. So $\sigma$ is not removed from $S_t$ and is therefore in $S_{t+1}$. It follows $S_{t+1}$ is also maximal. By induction it follows that $S_m$ is a maximal $m-$ordered set. Let $\sigma \in S_m$ and let $f \in E(G)\setminus \sigma$ and write $C(\sigma) = \{D_1, \dots, D_k\}$. If both ends of $f$ are in the same component of $\G(\sigma)$ then adding $f$ creates a cycle. If $f \in \nabla_G(\sigma) \subseteq \nabla_G(\pi)$ then $f = e_j$ for some $j$. But then because $S_m$ is $m-$ordered and $e_j \not\in \sigma$, we have $e_j$ dead in $\sigma$ and hence adding $f$ also creates a cycle. It follows that each $\G(\sigma)$, $\sigma \in S_m$ is maximally connected because there are no bridges in $E(G)\setminus \sigma$, hence $S_m$ contains only maximal elements in $B_G$. For a maximal element in $\sigma \in B_G$, an edge $e$ dead in $\sigma$ is equivalent to $e\not\in \sigma$. All maximal elements $\sigma \in B_G$ have the same number of connected components because they are maximally connected. So we have
\[\sum_{\substack{\sigma \supseteq \pi\\\sigma\in B_G}}(-1)^{\ell(\sigma)} = \sum_{\sigma \in S_0}(-1)^{\ell(\sigma)} = \dots = \sum_{\sigma \in S_m}(-1)^{\ell(\sigma)} = (-1)^{|C(G)|}|S_m|.\]
\end{proof}
Using the formula for the tree polynomial in terms of the chromatic polynomial lets us apply the above lemma.
\begin{theorem}
Let $G$ be a graph. The coefficient $[x^k]\tau_G$ is $(-1)^{|C(G)|+k}$ times the number of pairs $(E, T)$ where $E = \{e_1, \dots, e_m\}$ is a $k-$cutset, $e_1<\dots<e_m$ and $T$ is maximal in $B_G$ such that $e_1 \not\in T$ and for every $e_j \in E(T)$ there exists $i < j$ such that $e_j$ is in the fundamental cycle of $e_i$.
\end{theorem}
\begin{proof}
Note that $\binom{n-m}{k-m}[x^m]\chi_G$ is $(-1)^{|C(G)| + |V(G)|}$ times the pairs $(A, B)$ where $B \in B_G$ with $\ell(B) = m$ and $A \subseteq B$ with $\ell(A)=k$. We can then write
\begin{align*}
(-1)^{|V(G)|+k}\sum_{m=1}^{k}\binom{n-m}{k-m}[x^m]\chi_G &= (-1)^{|V(G)|+k}\sum_{\substack{\pi\in B_G\\\ell(\pi)=k}}\sum_{\substack{\sigma\in B_G\\\sigma\supseteq\pi}}(-1)^{\ell(\sigma)+|V(G)|}\\
&=(-1)^{k}\sum_{\substack{\pi\in B_G\\\ell(\pi)=k}}\sum_{\substack{\sigma\in B_G\\\sigma\supseteq\pi}}(-1)^{\ell(\sigma)}
\end{align*}
So by Lemma 4.4, $[x^k]\tau_G$ is $(-1)^{k+|C(G)|}$ times the number of pairs $(A, B)$ where $A \in B_G$ with $\ell(A) = k$, and $A \subseteq B \in B_G$ is maximal such that when writing $\nabla_G(A) = \{e_1, \dots, e_m\}$ (ordered), then $e_1 \not\in B$ and if $e_i \in B$ for $i > 1$ then there exists $j < i$ such that $e_i$ is in the fundamental cycle of $e_j$. This is in bijection with the pairs $(E, T)$ described in the theorem statement via the pair of mutually inverse bijective functions,
\begin{align*}
(A, B) &\mapsto (\nabla_G(A), B)\\
(E, T) &\mapsto (T\setminus E, T)
\end{align*}
\end{proof}

\begin{theorem}
Let $G$ be a graph with $n$ vertices. Then
\[\sum_{\lambda\vdash n}[X_{P_\lambda}]X_G = 1\]
\end{theorem}
\begin{proof}
For any graph $H$, $\sum_{\lambda\vdash m}[p_\lambda]X_H = \sum_{i}[x^i]\chi_H = 0$ if $H$ has edges, and $1$ if $H$ does have edges by Proposition 2.1. Write each $[p_\lambda]X_G = a_\lambda$. By \cite{STANLEY1995166} Theorem 2.9, we have $a_{1^n} = 1$ as the only $S \in B_G$ with $\lambda(S) = 1^n$ is $S = \emptyset$. We have $X_G = \sum_{\lambda\vdash n}a_\lambda p_\lambda$ so we have by Chromatic Reciprocity (\cite{aliniaeifard2021extended} Theorem 6.3),
\begin{align*}
\sum_{\lambda\vdash n}[X_{P_\lambda}]X_G &= \sum_{\lambda\vdash n}a_\lambda\sum_{\Delta\vdash n}[X_{P_\Delta}]p_\lambda\\
&=a_{(1^n)} + \sum_{\substack{\lambda\vdash n\\\lambda\ne 1^n}}a_\lambda\sum_{\Delta\vdash n}[p_\Delta]X_{P_\lambda}\\
&=a_{(1^n)} + \sum_{\substack{\lambda\vdash n\\\lambda\ne 1^n}}a_\lambda(0)\\
&=a_{(1^n)}\\
&= 1
\end{align*}
\end{proof}
This provides additional combinatorial information when used with Theorem 3.3. In general,
\begin{theorem}
Let $\{G_i\}_{i=1}^{\infty}$ be a family of connected graphs with each $G_i$ having $i$ vertices. Let $G$ be a graph with $n$ vertices. Then,
\[\sum_{\lambda\vdash n}[X_{G_\lambda}]X_G = 1\]
\end{theorem}
\begin{proof}
Write $X_G = \sum_{\lambda\vdash n}a_\lambda X_{G_\lambda}$. Note that
\begin{align*}
1 &= \sum_{\lambda\vdash n}[X_{P_\lambda}]X_G\\
&= \sum_{\lambda\vdash n}a_\lambda\left(\sum_{\Delta}[X_{P_\Delta}]X_{G_{\lambda}}\right)\\
&= \sum_{\lambda\vdash n}a_\lambda(1)\\
&= \sum_{\lambda\vdash n}a_\lambda
\end{align*}
\end{proof}
This similarly provides further combinatorial information for general chromatic bases derived from unweighted families of graphs.
\end{section}
\begin{section}{Tree Polynomials}
Define the function $\phi : \mathbb{Q}[x] \to \mathbb{Q}[x]$ by
\[\phi(p(x)) = \sum_{k=1}^{\deg(p(x))}(-1)^{n+k}\left(\sum_{m=1}^{k}\binom{n-m}{k-m}[t^m]p(t)\right)x^k\]
In particular, $\phi(\chi_G) = \tau_G$. 
\begin{proposition}For graphs $G, H$, $\phi(\chi_G\chi_H) = \phi(\chi_G)\phi(\chi_H)$\end{proposition}
\begin{proof}
\[\phi(\chi_G\chi_H) = \phi(\chi_{G\cup H}) = \tau_{G\cup H} = \tau_G\tau_H = \phi(\chi_G)\phi(\chi_H)\]
\end{proof}
It is important to note $\phi$ is not a homomorphism on all of $\mathbb{Q}[x]$. However, because it multiplies over chromatic polynomials, we may carry over some of their results to tree polynomials. For example, the well-known clique-gluing formula for the chromatic polynomial also applies to the tree polynomial.
\begin{cor}
Let $G_1, G_2$ be graphs containing a $k-$clique. Let $K_1 \le G_1$ and $K_2 \le G_2$ be $k-$cliques respectively and write $V(K_1) = \{u_1, \dots, u_k\}$ and $V(K_2) = \{v_1, \dots, v_k\}$. Define graph $H$ to have vertex set $V(G_1) \cup (V(G_2)\setminus V(K_2))$ and contain edges $u_iw \iff u_iw \in E(G_1)$ or $v_iw \in E(G_2)$ and the edges $xy$ with $x, y \in E(G_1)$ or $x, y \in E(G_2)$. That is, $H$ is the graph obtained by gluing $G_1, G_2$ at a $k-$clique, $K$. Then,
\[\tau_H = \frac{\tau_{G_1}\tau_{G_2}}{\tau_K}\]
\end{cor}
\begin{cor}
Let $G$ be a graph and let $H$ be $G$ with a tree glued to a vertex of $G$. Then $\tau_G = \tau_H$.
\end{cor}
The tree polynomial of a weighted graph also remembers the amount of excess weight in a graph (that is, the difference between the total weight and the number of vertices).
\begin{theorem}
Let $(G,\omega)$ be a weighted graph with $N$ total weight and $n$ vertices. Let $k$ be the largest integer such that $(x-1)^k|\tau_{(G,\omega)}$. Then $k = N-n$. Moreover, $\tau_{(G,\omega)} = (x-1)^k\tau_G$.
\end{theorem}
\begin{proof}
Proceed by induction on $N-n$. If $N-n = 0$ then $\tau_G(1)=1\ne 0$ by Theorem 3.4 and the result holds. Suppose the result holds for all weighted graphs with excess weight $m \in \mathbb{N}$ and suppose $N-n=m+1$. Pick a vertex $v_0$ with weight greater than $1$. Consider graph $G'$ with vertex set $V(G) \cup \{w\}$ and edgeset $E(G) \cup \{wv_0\}$. Let $\omega'$ be defined by $\omega'(w) = 1$, $\omega'(v) = \omega(v)$ for $v_0 \ne v \in V(G)$ and $\omega'(v_0) = \omega(v_0) - 1$. Note that $(G'/v_0w, \omega'/v_0w) = (G, \omega)$. So by weighted deletion-contraction (\cite{crew2020deletioncontraction} Lemma 2) and using the diagrams as stand ins for their chromatic symmetric functions,
\begin{center}
\begin{tikzpicture}
\draw[black] (5, 0) circle (20pt);
\node at (5, 0) {$(G,\omega)$};
\node at (1, 0) {$=$};
\draw[black] (2.5, 0) circle (20pt);
\node at (2.5, 0) {$(G,\omega')$};
\filldraw[black] (1.5, 0) circle (2pt);
\node at (3.75, 0) {$-$};
\draw[black] (0, 0) circle (20pt);
\node at (0, 0) {$(G,\omega')$};
\filldraw[black] (-1, 0) circle (2pt);
\draw[black, thick] (-1, 0) -- (-0.7, 0);
\node at (1.5, 0.5) {$w$};
\node at (-1, 0.5) {$w$};
\end{tikzpicture}
\end{center}
Taking tree polynomials and rearranging,
\[\tau_{(G,\omega)} = x\tau_{(G,\omega')} - \tau_{(G', \omega')}\]
By the inductive hypothesis, $\tau_{(G',\omega')} = (x-1)^{m}\tau_{G'} = (x-1)^{m}\tau_G = \tau_{(G, \omega')}$ where the second last equality follows from Corollary 4.2 (our construction of $G'$ is the same as gluing a path of length $2$ to $v_0$). So,
\[\tau_{(G,\omega)} = x\tau_{(G,\omega')} - \tau_{(G,\omega')} = (x-1)\tau_{(G,\omega')} = (x-1)(x-1)^m\tau_G = (x-1)^{m+1}\tau_G\]
and the result follows.
\end{proof}
There is also a deletion-contraction relation for the tree polynomial.
\begin{theorem}[Unweighted Deletion Contraction]
Let $G$ be a graph and $e$ an edge of $G$. Then
\[\tau_G = \tau_{G\setminus e} - (x-1)\tau_{G/e}\]
\end{theorem}
\begin{proof}
Follows from weighted deletion contraction of chromatic symmetric functions and Theorem 4.1.
\end{proof}
\begin{proposition}
$\tau_{p_\lambda} = x^{\ell(\lambda)}(x-1)^{|\lambda|-\ell(\lambda)}$
\end{proposition}
\begin{proof}
By chromatic reciprocity, each $[X_{P_\Delta}]p_\lambda = [p_\Delta]X_{P_\lambda}$. So
\[[x^k]\tau_{p_\lambda}=  \sum_{\Delta\\\ell(\Delta)=k}[X_{P_\Delta}]p_\lambda = \sum_{\Delta\\\ell(\Delta)=k}[p_\Delta]X_{P_\lambda} = [x^k]\chi_{P_\lambda}\]
Therefore,
\[\tau_{p_\lambda} = \chi_{P_\lambda} = x^{\ell(\lambda)}(x-1)^{|\lambda|-\ell(\lambda)}\]
\end{proof}
Definitions 4.2 and Proposition 4.1 means we can extend the chromatic polynomial to weighted graphs; however, unlike the tree polynomial, the excess weight is forgotten when taking the chromatic polynomial.
\begin{theorem}
Let $(G,\omega)$ be a weighted graph. Recall that $\chi_{(G,\omega)} := \chi_{X_{(G,\omega)}}$. We have
\[\chi_G = \chi_{(G,\omega)}\]
\end{theorem}
\begin{proof}
Let $G$ have $n$ vertices and total weight $N$. Proceed by induction on $N-n$. If $N-n=0$ we're done. Otherwise, pick vertex $v_0$ with weight $\omega(v_0) > 1$. Define graph $G'$ to have vertex set $V(G) \cup \{w\}$ and edgeset $E(G) \cup \{wv_0\}$. Define $\omega'(v) = \omega(v)$ if $v_0 \ne v \in V(G)$, $\omega'(w) = 1$ and $\omega'(v_0) = \omega(v_0) - 1$. Note that $(G'/wv_0, \omega'/wv_0) = (G, \omega)$. So by weighted deletion contraction of symmetric functions, 
\[X_{(G',\omega')}  = X_{(G'\setminus wv_0, \omega')} - X_{(G'/wv_0,\omega'/wv_0)}\]
Using diagrams as stand-ins for their chromatic symmetric functions,
\begin{center}
\begin{tikzpicture}
\draw[black] (5, 0) circle (20pt);
\node at (5, 0) {$(G,\omega)$};
\node at (1, 0) {$=$};
\draw[black] (2.5, 0) circle (20pt);
\node at (2.5, 0) {$(G,\omega')$};
\filldraw[black] (1.5, 0) circle (2pt);
\node at (3.75, 0) {$-$};
\draw[black] (0, 0) circle (20pt);
\node at (0, 0) {$(G,\omega')$};
\filldraw[black] (-1, 0) circle (2pt);
\draw[black, thick] (-1, 0) -- (-0.7, 0);
\node at (1.5, 0.5) {$w$};
\node at (-1, 0.5) {$w$};
\end{tikzpicture}
\end{center}
Rearranging,
\begin{center}
\begin{tikzpicture}
\draw[black] (0, 0) circle (20pt);
\node at (0, 0) {$(G,\omega)$};
\node at (1, 0) {$=$};
\draw[black] (2.5, 0) circle (20pt);
\node at (2.5, 0) {$(G,\omega')$};
\filldraw[black] (1.5, 0) circle (2pt);
\node at (3.5, 0) {$-$};
\draw[black] (5, 0) circle (20pt);
\node at (5, 0) {$(G,\omega')$};
\filldraw[black] (4, 0) circle (2pt);
\draw[black, thick] (4, 0) -- (4.3, 0);
\node at (1.5, 0.5) {$w$};
\node at (4, 0.5) {$w$};
\end{tikzpicture}
\end{center}
Each of the terms on the right-hand side has less excess weight. Taking power-sum basis polynomials on both sides yields
\[\chi_{(G,\omega)} = \chi_{(G'\setminus wv_0, \omega')} - \chi_{(G',\omega')}\]
By the inductive hypothesis and deletion-contraction for chromatic polynomials,
\[\chi_{(G,\omega)} = \chi_{G'\setminus wv_0} - \chi_{G'} = \chi_{G'/wv_0} = \chi_G\]
so the result follows by induction.
\end{proof}
Thus, the chromatic polynomial tells you how many vertices a graph has but not its vertex weights, whereas the tree polynomial can tell you how much excess weight a graph has, but not how many vertices (recall that $\tau_T = x$ for any tree $T$)! Hence these polynomials are in a natural sense dual with respect to the information they carry on vertex-weighted graphs.

As another link between the chromatic polynomials and tree polynomials, we have
\begin{theorem}
Let $(G,\omega)$ be a weighted graph with $n$ vertices and total weight $N$. Then,
\begin{enumerate}[label=(\arabic*)]
\item
$\chi_G = (x-1)^{N}\tau_{(G,\omega)}\left(\frac{x}{x-1}\right)$
\item
$\tau_{(G,w)} = (x-1)^{N}\chi_G\left(\frac{x}{x-1}\right)$
\end{enumerate}
\end{theorem}
\begin{proof}
Write $X_{(G,\omega)} = \sum_{\lambda\vdash N}a_\lambda X_{P_\lambda}$. Then taking $\chi$ of both sides,
\begin{align*}
\chi_{(G,\omega)} &= \sum_{\lambda\vdash N}a_\lambda \chi_{X_{P_\lambda}}\\
&= \sum_{\lambda\vdash N}a_\lambda x^{\ell(\lambda)}(x-1)^{|\lambda|-\ell(\lambda)}\\
&= \sum_{k=0}^{N}\left(\sum_{\substack{\lambda\vdash N\\\ell(\lambda)=k}}a_\lambda\right)x^k(x-1)^{n-k}\\
&= \sum_{k=0}^{N}[t^k]\tau_{(G,\omega)}(t)x^k(x-1)^{n-k}\\
&= (x-1)^N\sum_{k=0}^{N}[t^k]\tau_{(G,\omega)}(t)\left(\frac{x}{x-1}\right)^k\\
&= (x-1)^N\tau_{(G,\omega)}\left(\frac{x}{x-1}\right)
\end{align*}
Which shows $(1)$. Write $X_{(G,\omega)} = \sum_{\lambda\vdash n}b_\lambda p_\lambda$. Then,
\begin{align*}
\tau_{(G,\omega)} &= \sum_{\lambda\vdash N}b_\lambda \tau_{p_\lambda}\\
(\text{proposition }4.2)&= \sum_{\lambda\vdash N}b_\lambda x^{\ell(\lambda)}(x-1)^{|\lambda|-\ell(\lambda)}\\
&= \sum_{k=0}^{N}\left(\sum_{\substack{\lambda\vdash N\\\ell(\lambda)=k}}b_\lambda\right)x^k(x-1)^{N-k}\\
&= \sum_{k=0}^{N}[t^k]\chi_{G}(t)x^k(x-1)^{N-k}\\
&= (x-1)^N\sum_{k=0}^{N}[t^k]\chi_{G}(t)\left(\frac{x}{x-1}\right)^k\\
&= (x-1)^N\chi_{G}\left(\frac{x}{x-1}\right)
\end{align*}
Note that the last line follows because by Theorem 4.3, $[t^k]\chi_G(t) = 0$ for $k > n$. 
\end{proof}
Thus, given $\tau_{(G,w)}$ and the number of vertices of $G$ (or equivalently the total weight of $G$) we can recover $\chi_G$ and vice-versa. In particular, it follows that if $G$ is an unweighted graph with $n$ vertices and no edges that $\tau_G = \chi_G = x^n$.

A natural question is whether evaluating the tree polynomial at positive integers has a natural combinatorial meaning, as is true for the chromatic polynomial. From the computation of Theorem 4.4 (in particular, (2)), for a graph $G$ on $n$ vertices,
\[\tau_G(x) = \sum_{k=0}^{n}[t^k]\chi_G(t)x^k(x-1)^{n-k}\]
We can combinatorially interpret this formula. From Whitney's Broken Circuit Theorem \cite{whitney1932logical}, $[t^k]\chi_G(t)$ is $(-1)^{n-k}$ times the number of elements $E \in B_G$ with $C(E) = k$. As $\G(E)$ is a forest with $k$ components, $x^k(x-1)^{n-k}$ is the number of proper $x-$colourings of $\G(E)$ when $x$ is a positive integer. So we can interpret this sum as enumerating all properly $x-$coloured internal forests with a sign based on number of components:

\begin{theorem}
Let $x$ be a positive integer, and let $G$ be a graph with $n$ vertices. Define $S$ to be the set of properly $x-$coloured internal forests $F$ of $G$. Then
% \begin{enumerate}[label=(\arabic*)]
% \item
% The two ends of $\hat{e}$ are the same colour (so that $\hat{e} \not\in E(F)$).
% \item
% If $\hat{e} \ne e \not\in E(F)$ then either the two ends of $e$ are in the same component of $\G(F)$ or the two ends of $e$ are the same colour.
% \item
% If $e \in E(F)$ then $e$ is contained in the fundamental cycle of some non-monochromatic smaller edge $\overline{e} \not\in E(F)$.
% \end{enumerate}
Then,
\[\tau_G(x) = (-1)^n\sum_{F \in S}(-1)^{|C(F)|}\]
\end{theorem}
\end{section}
\begin{section}{Further Directions}

There are many other avenues of exploration for interpreting the tree polynomial. For a graph $G = (V, E)$ and a set partition $\pi = \{B_1, \dots, B_k\}$ of $V$, we say $\pi$ is a \emph{connected partition} if the induced subgraph of each $B_i$ by $G$ is connected. The \emph{lattice of contractions of $G$} (denoted $L_G$) is the poset of connected partitions of $V$ ordered by refinement (that is $\{D_1, \dots, D_\ell\} \le \{B_1, \dots, B_k\}$ if each block $D_i$ is contained in some $B_j$).  Note that the $\hat{0}$ element is the partition in which each vertex has its own block. We write $\ell(\pi) = k$ to denote the number of blocks in the set partition.
\begin{theorem}[\cite{whitney1932logical}]
The chromatic polynomial of a graph $G$ can be written as
\[\chi_G(x) = \sum_{\pi \in L_G}\mu(\hat{0}, \pi)x^{\ell(\pi)}.\]
\end{theorem}
With this result, we see that for a graph $G$ on $n$ vertices,
\begin{align*}
\tau_G(x) &= \sum_{k=0}^{n}[t^k]\chi_G(t)x^k(x-1)^{n-k}\\
&= \sum_{\pi \in L_G}\mu(\hat{0},\pi)x^{\ell(\pi)}(x-1)^{n-\ell(\pi)}\\
&= \sum_{\substack{\pi \ge \hat{0}\\\pi \in L_G}}\mu(\hat{0}, \pi)x^{\ell(\pi)}(x-1)^{n-\ell(\pi)}
\end{align*}
Which motivates the definition
\begin{definition}
Let $G$ be a graph and let $\pi \in L_G$. Then,
\[\tau_{G}(\sigma, x) := \sum_{\substack{\pi\ge\sigma\\\sigma\in L_G}}\mu(\sigma,\pi)x^{\ell(\pi)}(x-1)^{n-\ell(\pi)}.\]
\end{definition}
We note that $\tau_G(\hat{0}, x) = \tau_G(x)$. By M\"{o}bius Inversion,
\[x^{\ell(\pi)}(x-1)^{n-\ell(\pi)} = \sum_{\sigma \ge \pi}\tau_G(\sigma, x)\]
so we can interpret $\tau(G)$ in terms of $L_G$ if we can reverse engineer $\tau_G(\sigma, x)$. Another approach is to try to reverse engineer the unweighted deletion contraction relation (Theorem 5.2).

Theorem 3.4 provides us with new information regarding the coefficients when expanding a chromatic symmetric function over tree bases. Due to the fact that the stars and paths are the only families of unweighted graphs that satisfy a reciprocity relation with respect to the $p$-basis (\cite{aliniaeifard2021extended} Proposition 6.2), a further direction would be to find interpretations for the individual coefficients in the corresponding bases. A useful direction would be to try and find a more ``tame" simplification of the interpretation given in Theorem 3.3/Lemma 3.4. 

Finally, the contents of Theorem 4.4 involve fractional evaluations of the chromatic and tree polynomials. There are two ways to look at this; the first is that the formulas remind us of species composition. Is it then the case that some $B-$polynomials preserve plethystic composition? Another perspective that if tree polynomials are related to a certain family of fractional evaluations of the chromatic polynomial, does this suggest there are reasonable interpretations of other fractional evaluations of the chromatic polynomial?
\end{section}

\begin{section}{Acknowledgments}

The authors would like to thank Sophie Spirkl for helpful discussions and comments. 

This research was sponsored by the National Sciences and Engineering Research Council of Canada.

\end{section}

\nocite{*}
\printbibliography

\end{document}